\documentclass{amsart}
\usepackage{amsmath, amssymb, mathrsfs, amsthm, cases, verbatim}

\usepackage{tikz}

\usepackage[colorlinks, linkcolor=blue,  citecolor=blue, urlcolor=blue]{hyperref}%
\usepackage{url}
\usepackage{pdfsync}
\usepackage{xspace}
\usepackage{graphicx}
\usepackage{subfig} 
\usepackage{enumerate}
\usepackage{verbatim}
\usepackage{alltt}
\usepackage{amsthm}
\usepackage{color}

\graphicspath{{../figures/}}
\newtheorem{theorem}{Theorem}

\theoremstyle{remark}
\newtheorem{remark}{Remark}
\newtheorem{definition}{Definition}

\newtheorem{example}{Example}
\newtheorem{lemma}{Lemma}
\numberwithin{equation}{section}

\newcommand{\norm}[1]{\Vert#1\Vert}

\newcommand{\abs}[1]{\vert#1\vert}

\DeclareMathOperator{\median}{median}

\DeclareMathOperator{\dist}{dist}

\newcommand{\bq}{\begin{equation}}
\newcommand{\eq}{\end{equation}}
\newcommand{\R}{\mathbb{R}}

\newcommand{\bO}{\mathcal{O}}

\newcommand{\V}{{V}}
\newcommand{\E}{{E}}

\newcommand{\GG}{{Z}}
\newcommand{\Ww}{{W}}

\newcommand{\yy}{\mathbf{y}}
\newcommand{\xx}{\mathbf{x}}
\newcommand{\yx}{ \yy(x)}
\newcommand{\pp}{{p}}
\newcommand{\qb}{{q}}

\newcommand{\zb}{\mathbf{0}}
\newcommand{\wb}{\mathbf{1}}

\newcommand{\ur}{{u}}

\newcommand{\vr}{v}

\newcommand{\CG}{C(G)}
\newcommand{\Skr}{{F}}

\newcommand{\Sku}{{F(\ur)}}
\newcommand{\Skv}{{F(\vr)}}
\newcommand{\Skp}{{H}}

\newcommand{\grhs}{g}

\newcommand{\grad}{\nabla}

\newcommand{\Lap}{\Delta}
\newcommand{\MC}{\Delta_1}
\newcommand{\ME}{\lambda_1}

\newcommand{\Eikp}{\norm{\grad \ur(x)}^{+}}
\newcommand{\Eikn}{\norm{\grad \ur(x)}^{-}}

\usepackage[alphabetic]{amsrefs}
 \renewcommand{\div}{\operatorname{div}}

\begin{document}

\title[Nonlinear PDEs on graphs]
{Nonlinear elliptic Partial Differential Equations and p-harmonic functions  on graphs}

\author{J.~J.~Manfredi, A.~M.~Oberman, and A.~P.~Sviridov}
\address{Juan J. Manfredi
\hfill\break\indent
Department of Mathematics
\hfill\break\indent
University of Pittsburgh
\hfill\break\indent Pittsburgh, PA 15260, USA
\hfill\break\indent
{\tt manfredi@pitt.edu}}

\address{Adam M. Oberman
\hfill\break\indent
Department of Mathematics and Statistics
\hfill\break\indent
McGill University 
\hfill\break\indent Montreal, Quebec, Canada
\hfill\break\indent
{\tt adam.oberman@mcgill.ca}}

\address{Alexander P. Sviridov
\hfill\break\indent
Department of Mathematics
\hfill\break\indent
University of Pittsburgh
\hfill\break\indent Pittsburgh, PA 15260, USA
\hfill\break\indent
{\tt aps14@pitt.edu}}

\date{\today}

\begin{abstract}
In this article we study the well-posedness (uniqueness and existence of solutions) of nonlinear elliptic  Partial Differential Equations (PDEs) on a finite graph.     These results are obtained using the discrete comparison principle and connectivity properties of the graph.   This work is in the spirit of the theory of viscosity solutions for partial differential equations. 

The equations include the graph Laplacian, the $p$-Laplacian, the Infinity Laplacian,  and the Eikonal operator on the graph.

\end{abstract}
\thanks{ We would  like to thank the anonymous referees for the careful reading of the original version of this manuscript and their valuable comments and suggestions. The authors would also like to thank Yves van Gennip, Yao Yao, and Tiago Salvador for their careful reading of the revised version of this
manuscript.}

\maketitle


\section{Introduction}
In this article we consider existence and uniqueness of solutions to nonlinear elliptic  partial differential equations  on a finite graph.    
The uniqueness results are based on appropriate versions of the the discrete comparison principle, which for some equations depends on fine  connectivity properties of the graph.   

\subsection{Context and motivation}
The graph Laplacian has been of interest  since Birkhoff studied it  in the nineteenth century.  
 Finite difference discretizations of nonlinear elliptic Partial Differential Equations (PDEs)  have been used extensively in Image Processing~\cite{SapiroBook}.
Modern applications of PDEs on graphs include machine learning, clustering, and social networks~\cite{discretecalculusBook}. 

By their very definition, PDEs on graphs use only local information: locality is a requirement for problems on distributed networks~\cite{TsitlikisConsensus}.  
Very large graphs challenge the primarily combinatorial tools originally designed to study them~\cite{LovaszGraphs}.   
One popular measure of tractability is that algorithms run in polynomial time.  For our purposes, that is not sufficient.  
The PDEs of the types we study here can be solved quickly (in log-linear time) on grids~\cite{TsitsiklisFastMarching, FastSweeping,ObermanFroeseMANum}.   This kind of fast solvers may also be available for this class of equations on general graphs. 
 
Nonlinear PDEs on graphs arise naturally in stochastic control theory, when the state space is discrete~\cite{BertsekasBook1}. Numerical methods for stochastic control problems can be shown to converge using probabilistic methods~\cite{KushnerDupuisBook} or using the viscosity solutions theory~\cite{BSnum}.  
In general, wide stencil schemes are needed to represent linear equations by positive difference schemes~\cite{MotzkinWasow}.  Most discretizations assume a compact stencil, which leads to specific requirements on the structure of the equation, such as diagonal dominance~\cite{BSnum}.   Extending these schemes to general problems requires the use of wide stencils~\cite{Zidani}.
More recently, Lions and Lasry have also extended the HJB theory to Mean Field Games~\cite{LionsLasryMFG}. 
Mean Field Games have been posed on graphs~\cite{GueantMFGgraphs, GomesMFGgraphs}.    

 Stochastic control theory in the continuous setting leads to Hamilton-Jacobi-Bellman (HJB) equations~\cite{FlemingSonerBook}.  These equations motivated the theory of viscosity solutions.  
The theory of nonlinear (and possibly degenerate) elliptic PDEs in continuous space is now well understood~\cite{CIL}.   
In general, discretization of  these PDEs using the finite difference method may not converge.  However,  if the finite difference schemes also obey a comparison principle, then a convergence proof is available~\cite{BSnum}.  These finite difference schemes, which are called \emph{elliptic}~\cite{ObermanSINUM} can be characterized by a structure  condition.   Elliptic schemes generalize upwind schemes for first order equations and schemes of positive type for second order equations~\cite{MotzkinWasow}.  The PDEs studied here 
on general graphs coincide with elliptic finite difference schemes when the graph is a wide stencil finite difference grid.  Finite differences are discussed in Section  \ref{sec:finitediff}.

The Infinity Laplacian PDE~\cite{CrandallTour} has been well studied in both the discrete and continous settings.  
A convergent discretization of the  Infinity Laplace equations was presented in~\cite{ObermanILnum} (see also~\cite{LegruyerOlder}).  This discretization is reinterpreted here as a PDE on a graph.   This equation has also been interpreted using  random tug-of-war games, in both the continuous and the discrete setting~\cite{PSSW}.  
Variants of the Infinity Laplace equation can be posed on graph (see~\eqref{Fpdefn} below),  these also arise as  finite difference schemes for  $p$-harmonic functions~\cite{ObermanpLap}.
The uniqueness for Infinity Laplace equation follows from the uniqueness of the finite difference scheme, and from the comparison with cones property~\cite{ArmstrongSmartUniqueness, ArmstrongSmartFD, AntunovicPeresSheffieldSomersille}.   These results have been extended to more general equations in~\cite{Armstrong-Crandall-Julin-Smart} and to equations with drift in~\cite{Armstrong-Smart-Somersille}.
The existence and uniqueness of the solution of $p$-Laplacian  and the connection with the game interpretation has been futher studied in \cite{PS} and \cite{MPR1}.

An elliptic discretization for the equation for motion of level sets by mean curvature was presented in~\cite{ObermanMC}.   An earlier work~\cite{Catte} presented a convergent method, but this was semi-discrete: although the method was implemented on a grid, the scheme was presented in the continuous case.  The scheme in~\cite{ObermanMC} involved the median (see~\eqref{MC} below): the use of wide stencils was required to fully discretize the equation onto a grid.   The game interpretation of motion by mean curvature was presented in~\cite{KohnSerfaty}.   The game involved a formula similar to the one in~\cite{Catte}.   The median scheme and the game scheme are both consistent, so they yield very similar results in the continuous setting.  This game formulation does not have a natural graph interpretation because it requires a notion of direction\footnote{One player choses a direction vector, the other player choose whether to move in the direction of the vector or the opposite direction.}, which is not available on a graph.  The median scheme is defined on a graph below.    However,  there is no reason to expect in general that  solutions of the resulting flow on the graph respect the same properties (such as the shrinkage of level set curves) as do solutions of the continuous PDEs.

\subsection{Results of the article}
It is natural to ask whether there are general conditions on the PDE which lead to existence and uniqueness results.   
In this article, we establish well-posedness (uniqueness and existence of solutions) results for nonlinear elliptic PDEs using structure conditions on the equations and connectivity properties of the graph. 
We do not rely on linearity, variational structures,  game interpretations, or optimal control interpretations.  Instead, we use the comparison principle as the basic tool.
\par
We begin with the two simplest structure conditions.  
By analogy with~\cite{CIL} these correspond to proper and uniformly elliptic equations.  Existence and uniqueness results for proper equations were established in~\cite{ObermanSINUM}.   Our first result here is to define uniformly elliptic equations and to prove well-posedness.  In the uniformly elliptic case, the uniqueness result uses the idea of marching to the boundary, which can be found in the early paper of Motzkin-Wasow on linear elliptic finite difference schemes~\cite{MotzkinWasow}.  This proof readily generalizes to the nonlinear case.  Well-posedness can fail for elliptic PDEs on graphs, as examples below show.
Our results are specific enough to avoid these examples, while still being general enough to consider a wide class of operators.

The next class of operators we consider are generalizations of the Eikonal and Infinity Laplace equations.
These are degenerate (neither proper nor uniformly elliptic) equations which depend  on the maximum and the minimum of the neighboring values.  The uniqueness proof for the generalized Infinity Laplace equations is a variation of the proof in~\cite{LegruyerOlder}, see also~\cite{ArmstrongSmartUniqueness} and~\cite{ArmstrongSmartFD}.  Another proof based on martingales is in~\cite{APS}. 


\begin{remark}
We do not discuss parabolic equations, but the theory can easily be modified to include this case~\cite{ObermanSINUM}.  The focus of \cite{ObermanMC, ObermanILnum, ObermanCENumerics} was on numerical approximation.  Well posedness for the discrete equations was not established.  This result was not  needed  for convergence, because a small perturbation  of the equation makes it proper without affecting consistency.
\end{remark}

\section{PDEs on weighted graphs, definitions, properties and examples}
We consider a finite weighted directed graph $G=(\V, \E, w)$.  
Here $\V$ is the set of vertices $\{x_1, \dots x_N\}$ and $\E \subset \V\times \V$ is the set of oriented edges.  We denote by $e_{xy}$ the edge $(x,y)$ that goes from the vertex $x$ to the vertex $y$.  The positive weight function $w\colon \E \to \R^+$ is defined on directed edges that join different vertices. We write
\[
w(e_{xy}) =  w_{xy}>0
\]
for the value of the weight function on the edge $e_{xy}$.

We also identify a non-empty subset  $\partial \V \subset \V$, which we call the \emph{boundary} of the graph. The interior of the graph is the set of vertices $V\setminus \partial V$. 
We write $x,y$ for typical vertices in $\V$. The vertex $y\not= x$ is a neighbor of the vertex $x$,  if there is an edge $e_{xy}$ from $x$ to $y$. The degree of a vertex  $d(x)$ is the number of neighbors of the vertex $x$.   We shall always assume that for an interior vertex $x\in \V\setminus \partial \V$ we have
$d(x)\ge 1$.
For each $x \in \V$ we fix an ordering of the neighbors of $x$, 
which we write as
\[
\yx = (y_1, y_2, \ldots, y_{d(x)}).
\]
This is merely for notational convenience because our results will be independent of the choice of ordering.  We also define the set formed by the vertex $x$ and its neighbors
$N(x)=\{x, y_1,y_2,\ldots, y_{d(x)}\}$. \par
The directed distance between $x$ and its neighbor $y\not = x$ is 
\[
d(x,y)=\frac{1}{w_{xy}}. 
\]
We also set $d(x,x)=0$. 
The distance between two arbitrary vertices $x$ and $y$ is the minimal path distance
\[
d(x,y)=\inf\Big\{d(x,x_1)+d(x_1,x_2)+\ldots +d(x_{k}, y) \Big\},
\]
where the infimum is taken over all finite paths proceeding via neighboring 
vertices which start at $x$ and end at $y$. If there is no path connecting
$x$ and $y$ we set $d(x,y)=+\infty$.  We say $x$ is connected to $y$ if $d(x,y)$ is finite.    We say that the graph is connected to the boundary, if for any $x\in \V$ there is some $y\in \partial \V$ with $d(x,y) < \infty$.

\begin{example}\label{geomG}
We define geometric graphs, such as those pictured in \autoref{fig:geometricgraphs}.  
In a geometric graph the edge relations are determined by  a set of 
$l$ non-zero linearly independent vectors 
$\{v_1, v_2,\ldots, v_l\}$ in $\mathbb{R}^n$, where $l\le n$, as follows. Set
\[
\mathbf{H}= (v_1, -v_1, v_2, -v_2, \ldots, v_l, -v_l).
\]
For an interior vertex $x$, the neighbors of $x$ are given by
\[
\yy(x)=x+\mathbf{H}.
\]
A boundary vertex is a vertex that does not have  full set of neighbors. 
A compact geometric graph (as a metric space) is necessarily finite.  A \emph{convex} geometric graph is defined by the property that for  any $x\in V$, if we have that $ y = x + mv_j \in V$ for a natural number $m$, we also have
\[
x + nv_j \in V, \quad n = 1, 2, \dots m-1. 
\]
A regular convex geometric graph is one that has at least one  interior vertex; in other words, for some $x\in V$ we have  $x+ \mathbf{H} \subset V$.
The interior of $V$  is the set of interior vertices, and the boundary $\partial V$ is its complement, which consist of precisely those vertices which do not have a full set of neighbors. The symmetric weights for adjacent vertices are given by inverse of the Euclidean distances between vertices
\[ 
w(x, x\pm v_j) = 1/\norm{v_j}.
\]
\end{example}

\begin{figure}
\begin{center}
\scalebox{.55}{
\includegraphics{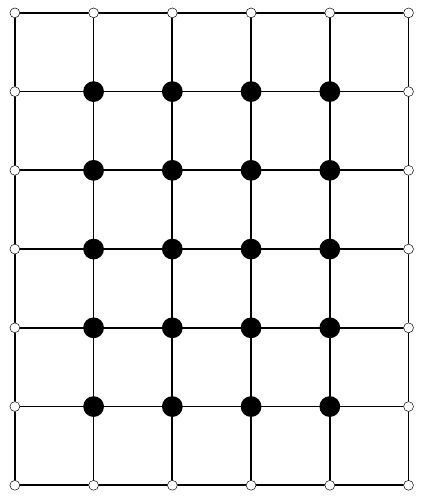}
\includegraphics{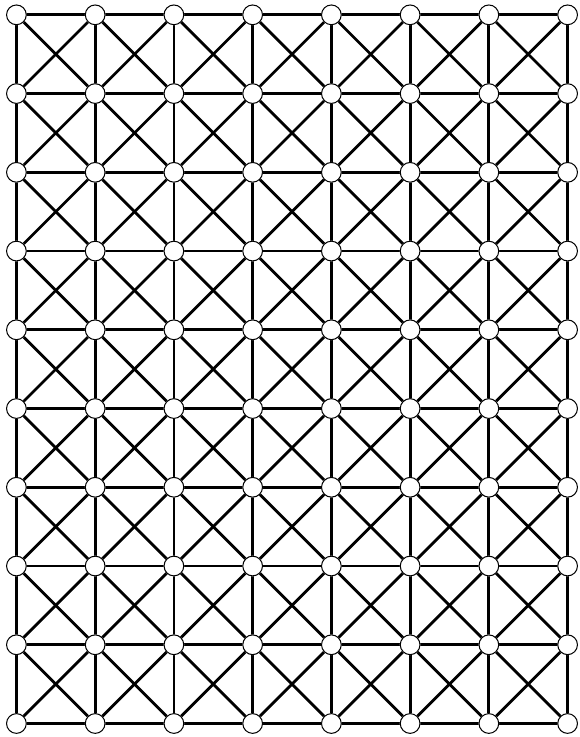}
\scalebox{.85}{\includegraphics{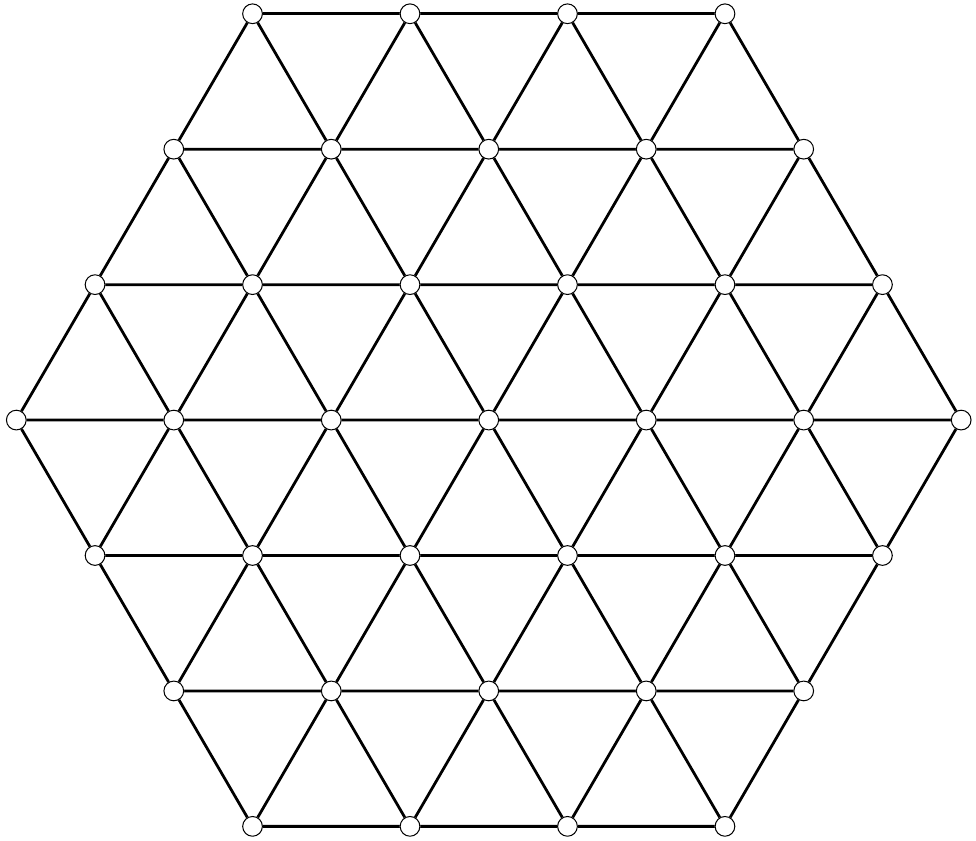}}
}
\end{center}
\caption{Regular convex geometric graphs: the boundary consists of the extremal vertices.}
\label{fig:geometricgraphs}
\end{figure}

We generally use boldface type  to emphasize that a quantity is a vector.
\begin{definition}[Vector Order Relations]\label{defn:vec}
Given vectors $\mathbf{p}, \mathbf{q} \in \R^n$ we write
\[
\mathbf{p} \le \mathbf{q} \iff  p_i \le q_i, \quad i=1,\dots, n.
\]
Also we write 
\[
\mathbf{p} < \mathbf{q} \iff \mathbf{p}\le \mathbf{q} \text{ and } p_i < q_i \text{ for some } i.
\]
We also write $\wb=(1,1,\ldots,1)$,  where the number of entries will be clear from context. 
\end{definition}

%

\subsection{PDEs on graphs}The class of all functions $u\colon V\to \R$ will be denoted by $C(G)$. 
The tangent space $TG(x)$  to the graph  $G$ at an interior vertex $x$ is
\[
TG(x)  = \R^{d(x)},
\] 
where $d(x)$ is the degree of $x$. The tangent bundle $TG$  is the disjoint
union of the tangent spaces
\[ 
TG=\bigcup_{x\in \V\setminus\partial \V} TG(x).
\]
Given
a function $u\in C(G)$ we write 
\[
u(\mathbf{y}(x)) =(u(y_1), \dots,  u(y_{d(x)})),
\]
for the list of values of $\ur$ at the list of neighbors $\yx$ at interior vertices
$x\in V\setminus\partial V$.  
\begin{definition}[Gradient]  
The gradient vector operator 
 $$\grad\colon \CG \mapsto TG$$  acting on a function $u\in C(G)$ 
 associates an element of $TG(x)$ to each interior vertex $x$  as follows
 \bq\label{gradient}
\begin{aligned}
\grad \ur(x) 
&=
\left( w(e_{xy{_1}})\left (\ur(y_1) - \ur(x) \right ) , \dots, w(e_{xy{_d}})(\ur(y_{d}) - \ur(x)) \right).
\end{aligned}
\eq
\end{definition}

We consider a special class of operators $$\Skr: \CG \to \CG,$$  called \emph{Partial Differential Equation} (PDE) operators. These operators are local in the sense that the value of $\Sku(x)$ at an interior vertex $x$ depends only on the value of the function $\ur$  and its gradient $\grad u$ at $x$. 

\begin{definition}
We say that  $\Skr: \CG \to \CG$ is a partial differential equation  on the graph $G$, if it can be written in the form
\bq\label{PDE}\tag{PDE}
\left\{
\begin{aligned}
F(u)(x)&=f(x,\ur(x),\grad u(x)) & \text{     for   }x \in \V \setminus \partial \V,  \\
F(u)(x)&= g(x) - \ur(x)   & \text{ for }x\in \partial \V,
\end{aligned}
\right.
\eq
where $f\colon(\V\setminus\partial\V) \times \R \times TG \to \R$
and $g\colon\partial V\to \R$.
\par
The \emph{Dirichlet problem} for our PDE requires finding a function 
$\ur\in \CG$ such that 
\bq\label{D}\tag{D}
\left\{
\begin{aligned}
f(x,\ur(x),\grad u(x)) &= 0, & \text{     for  }x \in \V \setminus \partial \V,  \\
g(x) - \ur(x) &=0,  & \text{ for }x\in \partial \V.
\end{aligned}
\right.
\eq
\end{definition}
 
Note that for $u$ to be  a solution to the Dirichlet problem \eqref{D}  $\Sku(x) = 0$ must hold for all $x\in \V$. 
\section{Examples of PDEs on graphs}

\subsection{Specific Examples}\label{sec:genform}
Consider a finite weighted directed graph $G=(\V,E,w)$.  In the following examples, where the   condition on $\partial V$ is not explicitly mentioned, we assume that $x\in \V\setminus\partial \V$ is an interior vertex.\par
The most studied PDE on a graph is the graph Laplacian. 
\begin{definition} The (weighted) graph Laplacian in  $G$ is given by
\bq\label{LapDefn}
\Delta u(x) =   \wb \cdot \nabla u(x) 
= \sum_{i=1}^{d(x)}  w_{xy_i} (\ur\left(y_i\right)- \ur(x)).
\eq

\end{definition}

We define the  component-wise maximum and minimum functions  on vectors  $\xx \in \R^d$ 
by
$\min(\xx) =\min \{ x_1,\dots, x_d \}$ and 
$\max(\xx) = \max\{ x_1,\dots, x_d \}$. This leads to two
 natural  PDEs   related to the  graph distance known as the eikonal operators.
\begin{definition} The positive and negative eikonal operators on the graph are 
\bq\label{EikPDefn}
\Eikp  =  \max (\grad u(x)) =    \max_{1\le i\le d(x)} w_{xy_i}  \left ( \ur(y_i) - \ur(x) \right )
\eq
and 
\bq\label{EikNDefn}
\Eikn  =   \min (\grad u(x)) = \min_{1\le i\le d(x)}   w_{xy_i}  \left ( \ur(y_i) - \ur(x) \right )
\eq
respectively,
where we have used definition~\eqref{gradient}.
\end{definition}

\begin{example}\label{ex:eikonal}
Consider the homogeneous Dirichlet problem~\eqref{D} for the positive eikonal operator
\[
\left \{
\begin{aligned}
 \Eikp - 1&=0, &&  x \in \V\setminus \partial \V,
\\
-\ur(x)&=0, && x \in \partial \V.
\end{aligned}
\right .
\]
The solution is the negative distance function to the target set $\partial \V$.  The positive distance function is the solution of the corresponding Dirichlet problem for  the negative eikonal operator  $\Eikn +1 =0$. Let us verify the last statement.
Given $\ur(x)$ equal to the (positive) distance function.  Then the minimal component of $\grad u$ will  be  a negative one, along the neighbor which contains the shortest path to the target set. Because the path from this neighbor will be one vertex
shorter, the  equation holds.   We show uniqueness below in~\autoref{sec:uniqueEik}.
\end{example} 

\begin{definition} The infinity Laplacian on the graph is given by
\bq\label{ILDefn}
\Delta_{\infty}\ur(x)=
\frac{\Eikp +  \Eikn}{2}.
\eq
\end{definition}

\begin{definition}
The  1-Laplacian operator on the graph is given by
\bq\label{MC}
\MC \ur(x) = {\median} (\grad u(x)).
\eq
The median of the set $\{x_1,\dots, x_k\}$  is found by arranging all the numbers from lowest value to highest value and selecting the middle one. If $k$ is even, the median is defined to be the mean of the two middle values.
\end{definition} 
This operator is studied in~\cite{ObermanMC}, 
where it is shown to correspond to the operator in the equation for motion of level sets by mean curvature, provided the neighbors are arranged close to uniformly about a circle of small radius centered at $x$.

\subsection{Definitions: Elliptic PDEs and the Comparison Principle}

\begin{definition}[Comparison principle]
Given a graph $G$ and a partial differential equation $\Skr$, we say that \emph{the comparison principle} holds for $\Skr$ in $G$ if we have
\bq\label{comparison}\tag{Comp}
\Skr(u)(x) \ge \Skr(v)(x)\text{  for  all } x\in V \implies u(x) \le v(x) 
\text{ for all } x\in V.
\eq

\end{definition}
Uniqueness of solutions clearly follows from comparison, since if $u$ and $v$ are solutions $\Skr(u)(x)=\Skr(v)(x)=0\text{  for  all } x\in V$, so that  $u(x) \le v(x) $ and $u(x) \ge v(x)$. Thus  we conclude $u(x)=v(x)$ for all $x\in V$.

\begin{remark}
For the Dirichlet problem~\eqref{D} the inequality
$\ur(x) \le \vr(x)$ for $x\in \partial \V$ follows immediately from 
$\Skr(u)(x) \ge \Skr(v)(x)\text{  for } x\in \partial V$,  so  that to establish the comparison principle we would need to show that   $\ur(x) \le \vr(x)$ for interior vertices $x\in \V\setminus \partial \V$.
\end{remark}

The following definition refers to properties of the PDE $F$ on interior vertices
$V\setminus\partial V$.
\begin{definition}\label{defn:degenElliptic}    
The PDE  $\Skr$ is \emph{elliptic} at the vertex $x\in V\setminus\partial V$ if 	 we have
\bq\label{elliptic}
r \ge s, \quad \mathbf{\pp} \le \mathbf\qb \implies f(x, r, \mathbf{\pp}) \le f(x,s,\mathbf{\qb}). 
\eq for all $r,s\in \R$ and  all $\mathbf{p},\mathbf{q}\in \R^{d(x)}$.\par
The PDE  $\Skr$  is elliptic  in $G$ if it is elliptic for all $x\in \V\setminus\partial V$.\par
The PDE  $\Skr$ is \emph{proper} at the vertex $x\in V\setminus\partial V$ if
\bq\label{proper}
r > s, \quad \mathbf{\pp} \le \mathbf\qb  \implies f(x, r, \mathbf{\pp}) <f(x,s,\mathbf{\qb}).
\eq or all $r,s\in \R$ and  all $\mathbf{p},\mathbf{q}\in \R^{d(x)}$.\par
The PDE  $\Skr$  is proper in $G$ if it is proper for all $x\in \V\setminus\partial V$.\par
The PDE  $\Skr$  is \emph{uniformly elliptic} at the vertex  $x\in \V\setminus\partial V$ if we have
\bq\label{Uelliptic}
r \ge s, \quad \mathbf{\pp} < \mathbf\qb \implies f(x, r, \mathbf{\pp}) < f(x,s,\mathbf{\qb}).
\eq
for all $r,s\in \R$ and  all $\mathbf{p},\mathbf{q}\in \R^{d(x)}$.\par
Recall that  strict inequality means strict in at least one component.\par
The PDE  $\Skr$ is uniformly elliptic in $G$ if it is uniformly elliptic for all  $x\in \V \setminus \partial \V$.
\end{definition}

\begin{example}
The graph Laplacian is uniformly elliptic while the trivial equation $\Sku(x) = \grhs(x) - \ur(x) $ for all $x\in V$ is proper.   The eikonal equation and the infinity Laplace equation are elliptic, but they are neither proper nor uniformly elliptic.
\end{example}

Observe that if the function $u$ has a local maximum at the vertex  $x\in V\setminus \partial V$  we  have 
\bq\label{localmax}
\grad u(x) \le \zb.
\eq
If the local maximum is strict, the condition becomes
\bq\label{strictlocalmax}
\grad u(x) < \zb.
\eq


\begin{remark}
The condition~\eqref{localmax} plays the role of the familiar second derivatives
condition  $D^2 u(x) \le 0$ at a local maximum $x$  of  $u$,  when  $u$ is twice-differentiable.
\end{remark}

The following lemma is a restatement of Definition~\ref{defn:degenElliptic} in terms of local maxima of $\ur - \vr$.
\begin{lemma}\label{lemma1}Let $x\in V\setminus\partial V$ be an interior vertex. 
The PDE $\Skr$ is elliptic at the vertex  $x$ if we have
\begin{equation}\label{Fell}
\ur(x) \ge \vr(x)\text{ and }  \grad u(x) \le \grad v(x)
\implies
\Sku(x) \le \Skv(x).
\end{equation}
The PDE $\Skr$  is proper at the vertex $x$ if
\begin{equation}\label{Fellp}
\ur(x) > \vr(x) \text{ and } \grad u(x) \le \grad v(x)
\implies
\Sku(x) < \Skv(x).
\end{equation}
The PDE $\Skr$ is uniformly elliptic at the vertex $x$ if
\begin{equation}\label{Fellu}
\ur(x) \ge \vr(x) \text{ and }  \grad u(x) < \grad v(x)
\implies
\Sku(x) < \Skv(x).
\end{equation}
The hypothesis of (\ref{Fell}) is precisely that $u-v$ has a non-negative local maximum at the vertex $x$, while in (\ref{Fellp}) the hypothesis is that $u-v$ has a strictly positive local maximum, and in (\ref{Fellu}) the hypothesis is that $u-v$ has a non-negative strict local maximum. 
\end{lemma}

\begin{remark}
These definitions correspond in the continuous case  to the definitions of
proper, elliptic, and uniformly elliptic equations  in \cite{CIL}, with the inequalities reversed.  The reversal of inequalities comes from the fact that some authors, us included,   use the convention that $\Lap u$ is elliptic, which others use the convention that $-\Lap u$ is elliptic. 
 \end{remark}

\begin{remark}
Proving the converse of the inequalities in Lemma \ref{lemma1} leads to uniqueness for solutions of the PDEs.  We will prove that solutions of uniformly elliptic and proper PDEs are unique.   We will also give other conditions on (degenerate) elliptic schemes which lead to uniqueness.
\end{remark}

\subsection{Generalized $p$-harmonious equations} We
consider elliptic PDEs  that are expressed in terms
of functions $f(x, \mathbf{p})$  for interior vertices $x$ and that are independent of $u$.
\begin{definition} The elliptic function $f(x,\mathbf{p})$ is homogeneous if there exists a non-negative function $w_0\colon V\setminus \partial V\to \mathbb{R}$ such that we have 
\bq\label{fprop2}
w_0(x)\, {{\min}(\mathbf{p}) }\le f(x,\mathbf{p}) \le w_0(x) \, {\max(\mathbf{p})},
\eq
for all $x\in V\setminus \partial V$ and $\mathbf{p}\in\mathbb{R}^{d(x)}$. 
\end{definition}

\begin{definition}
We say $\Skr$ is  a \emph{(generalized) $p$-harmonious equation} if it is of the form
\bq\label{pharm2}
\Sku(x)  = w_+(x) \Eikp  + w_-(x)  \Eikn + f(x,\grad u(x)),
\eq
where $w_+$ and $w_-$ are strictly
positive functions in $V\setminus\partial V$, and the   function $f$ is  elliptic and homogeneous as in ~\eqref{fprop2}. 
\end{definition}
\begin{example}
Consider the special  case of~\eqref{pharm2}  given by
\bq\label{Fpdefn}
\Skr(u)(x) = w_0 \,  \Delta \ur(x) + w_1 \MC \ur(x)
 + w_+\Eikp  + w_-  \Eikn,  \eq
where $w_0$, $w_1$,  $w_+$, and $w_-$ are non-negative functions on interior vertices.
We call a solution of $F(u)(x)=0$ with $F$ given by \eqref{Fpdefn}  $p$-harmonious with drift by analogy with the continuous case in \cite{MPR2}.   Functions of this type arise as the approximations of $p$-harmonic functions~\cite{ObermanpLap, MPR2} when $w_+=w_-$.\par

 The  classical (not on a graph) normalized $p$-harmonic operator for $1\le p <\infty$ is 
\[
\begin{split}
\Delta_{p}^N u=\frac{1}{p}|\nabla u|^{2-p}\Delta_{p} u=\frac{1}{p}|\nabla u|^{2-p}\ {\div}(|\nabla u|^{p-2}\nabla u),
\end{split}
\]
while  for $p=\infty$ we set
\bq\label{infinitylap}
\begin{split}\Delta_{\infty}^N u=|\nabla u|^{-2}\Delta_{\infty} u=|\nabla u|^{-2}\ \langle D^2u\, \nabla u,\nabla u\rangle
\end{split}.
\eq
It is easy to see that
\begin{equation}
\label{eq:norm-p-lap-in-parts}
\Delta_{p}^N u= \frac{1}{p}\Delta_{1}^N u + \frac{1}{q} \Delta_{\infty}^{N}u,
\end{equation}
where $q$ is the H\"older conjugate of $p$, $1/p+1/q=1$. Note that
$\Delta_{2}^{N}u= (1/2) \Delta u$ and that
$$ \Delta u= \Delta_{1}^{N}u+ \Delta_{\infty}^{N}u.$$ 
Therefore if we set $\omega_0=1$, $\omega_1=1/p$, and $w_+=w_-=1/(2q)$
in (\ref{Fpdefn}) we obtain the discrete analogue of the normalized $p$-Laplace operator
\bq\label{Fpdefn-new}
F_p(u)(x) =   \frac{1}{p} \MC \ur(x)
 + \frac{1}{2q}\Eikp  + \frac{1}{2q} \Eikn. \eq

 In the latter case
an approximating sequence is generated by running  tug-of-war stochastic games with noise  of decreasing step-size.  
The value of the game function satisfies a nonlinear equation, which is directly linked to the existence and uniqueness of the solution to the $p$-Laplacian for $1<p\le \infty$,  as demonstrated in  \cite{PSSW}, \cite{PS} and \cite{MPR2}.

\end{example}
\begin{definition}
We say $\Skr$ is a \emph{positive eikonal}  equation if it is of the form
\bq\label{eikp1}
\Sku(x) = w_+(x) \Eikp + f(x,\grad u(x)) - g(x), 
\eq
where $w_+$ and $g$ are strictly positive functions, and  the 
  function $f$ is elliptic and homogeneous as in ~\eqref{fprop2}. 
\end{definition}

\section{Proofs: uniformly elliptic and proper cases}
The uniqueness proofs for different classes of PDEs share the same structure:  (i) assume a failure of the comparison principle, (ii) identify a  set where the comparison principal fails maximally, and (iii) use the properties of the equation to propagate this set to the boundary and obtain a contradiction. To this end we give the following definition.

\begin{definition}\label{defn:M}
Given $\ur,\vr \in \CG$ define the set of vertices where maximum of $u-v$ is attained
\[
\Ww = \Ww(\ur,\vr) = \arg \max_{x\in \V }(\ur(x) -\vr(x)) 
\]
and the actual  maximum 
\[
M = M(\ur,\vr) =\max_{x\in \V }(\ur(x) -\vr(x)). 
\]

\end{definition}

\subsection{Counterexamples to well-posedness}

\begin{example}[Failure of existence]
Existence can fail for the Dirichlet problem for an eikonal equation with the wrong sign, e.g.~$\norm{\grad u}^+ = -1$. 
 A simple example is $K_3$, the fully connected graph with three vertices 
 $\V = \{ A,B, C \}$, three edges $e_{AB}$, $e_{BC}$, and $e_{CA}$,  and unit weights.  
 Let $C$ be a boundary point, with $u(C) = 0$.    
 The PDE can be written after a simple manipulation as 
 \[
 u(A) = \max\{u(B),0\}+ 1, 
 \quad
u(B) = \max\{u(A),0\} + 1,
\]
which has no solution, since the first equation implies $u(A) > u(B)$, and the second $u(B) > u(A)$.
\end{example}

In the continuous setting, the PDE   for level set motion by mean curvature  is well-posed in the parabolic case, but uniqueness fails for the Dirichlet problem, see~\cite{SZ}. Uniqueness also fails for the Dirichlet problem in the discrete case, as the following example shows.

\begin{example}[Failure of uniqueness]\label{Failure of uniqueness}

Uniqueness can fail for the 1-Laplacian~\eqref{MC}.
Consider the planar geometric graph $G$ illustrated in \autoref{fig:fourvertices}, where  the weights are set equal to one. 
The twelve vertices in $G$ are $$\V=\bigg\{ (i,j), \, i, j\in\{-2,-1,1,2\} \colon  |i|+|j|\le 3\bigg\}.$$   We set $\partial\V=\left\{ (i,j) \in \V \colon  |i|+|j|= 3\right\}$ (eight vertices). The edges are given by segments parallels to the axis in such a way that each for the four vertices in
$\V\setminus\partial\V$ has four neighbors. We assign boundary values 
by setting 
\[
g(1,2)=g(-1,2)=g(-1,-2)=g(1,-2)=-1
\]
and 
\[ 
g(2,1)=g(2,-1)=g(-2,1)=g(-2,-1)=1. 
\]
Then,  the constant functions 
\[
u(i,j)=1,  \quad (i,j)\in\V\setminus\partial\V
\]
and
 \[
v(i,j)=-1,  \quad (i,j)\in\V\setminus\partial\V
\]  
are both solutions.  This is because in both cases, three of the four neighbors have the same value.\end{example}

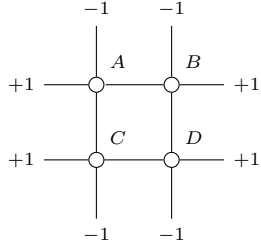
\begin{figure}

\begin{tikzpicture}
\tikzstyle{every node}=[font=\scriptsize]

\node (13) at (1,3) {$-1$};
\node [circle,draw,inner sep=0pt,minimum size=2mm,label=60:$A$] (A)  at (1,2) {} edge [-] (13);
\node [circle,draw,inner sep=0pt,minimum size=2mm,label=60:$C$] (C)  at (1,1) {} edge [-] (A);
\node (10) at (1,0) {$-1$} edge [-] (C);

\node (23) at (2,3) {$-1$};
\node [circle,draw,inner sep=0pt,minimum size=2mm,label=60:$B$] (B)  at (2,2) {} edge [-] (23) ;
\node [circle,draw,inner sep=0pt,minimum size=2mm,label=60:$D$] (D)  at (2,1) {} edge [-]  (B) edge [-] (C);
\node (20) at (2,0) {$-1$} edge [-] (D);
\node (01) at (0,1) {$+1$} edge [-] (C);
\node (02) at (0,2) {$+1$} edge [-] (A);
\node (31) at (3,1) {$+1$} edge [-] (D);
\node (32) at (3,2) {$+1$} edge [-] (B);
\node (23) at (1,2) {} edge [-] (B);
\end{tikzpicture}
\caption{The graph and the boundary values.}
\label{fig:fourvertices}
\end{figure}

\begin{remark}  
Uniqueness fails for the non-homogeneous infinity Laplacian equation
\[
\Delta_{\infty} \ur = g,
\]
in the case where $g(x)$ changes sign as noted in \cite{PSSW}.  For failure of uniqueness in the continuous case 
see~\cite{ArmstrongSmartFD}.
\end{remark}

\subsection{Proofs}



\begin{theorem}
Consider a graph $G$ which is connected to the  boundary $\partial \V\subset \V$. 
Suppose $\Skr$  is uniformly elliptic,  then the comparison principle \eqref{comparison} holds   for the Dirichlet problem~\eqref{D}.
\end{theorem}

\begin{proof} Suppose that $F(u)(x)  \ge F(v)(x)$ for all $x\in V$. 
It follows that $u(x)\le v(x)$ for boundary vertices $x\in\partial V$.
We wish to show that $u(x) \le v(x)$ for interior vertices $x\in V\setminus\partial V$.
Define the set $\Ww$ and the quantity $M$ as in Definition~\ref{defn:M}.
Assume that $M>0$. By assumption, there is a path from any point in $\Ww$ to some point in $\partial \V$, 
so by following the path until the first point outside the set is found, 
there exists $x \in \Ww$ with neighbor $y$ such that $y \not\in \Ww$.  We conclude that 
$x$ is a strict positive local maximum of $\ur-v$, so that
$\nabla u(x) < \nabla v(x)$.  Since $\Skr$ is uniformly elliptic at $x$, we can apply~\eqref{Fellu} to obtain 
\begin{align*}
\Sku(x) < \Skv(x),
\end{align*}
which contradicts the uniform ellipticity assumption and completes the proof.
\end{proof}

\begin{theorem}
Consider a graph $G$ which is connected to the  boundary $\partial \V\subset \V$. 
Suppose $\Skr$  is proper,  then the comparison principle \eqref{comparison} holds   for the Dirichlet problem~\eqref{D}.
\end{theorem}
\begin{proof}
The proof is similar to that of previous theorem. In this case we take $x\in W$, a positive local maximum of $u-v$ so that $\nabla u(x)\le \nabla v(x)$. 
Because $u(x) > v(x)$ and $\Skr$ is proper we deduce
$$ F(u)(x)=f(x, u(x),\nabla u(x)<f(x,v(x),\nabla v(x))=F(v)(x),$$
in contradiction with our hypothesis.
\end{proof}

\begin{remark} We thank the anonymous referee for pointing out that indeed what  is used in
the proofs of the above theorems is the condition
$$r>s, \,\, \mathbf{p}< \mathbf{q} \implies f(x,r, \mathbf{p})< f(x,s, \mathbf{q})$$
for interior vertices $x\in V\setminus\partial V$. This conditions is weaker
than both uniform ellipticity and properness.\end{remark}

\section{Uniqueness for Eikonal and $p$-harmonic Equations}
In contrast to the uniformly elliptic case, where~\eqref{Fellu} holds,  for the equations considered in  this section,  a strict local maximum  of $\ur-v$ does not imply a strict inequality for the values of the equation.  Instead, we need to locate a neighboring value which is active (see Lemma \ref{lemHP} below)  in the equation.

\subsection{A Lemma on Propagation of Maxima}
We begin with the definition of an active neighbor.

\begin{definition} Given an interior vertex $x$ and a neighbor $z$ of $x$, we say  $z$ is an \emph{active} neighbor of $x$ if 
\[
\Eikp = w_{xz} ( \ur(z) - \ur(x)).
\] 
\end{definition}

\begin{lemma}\label{lemHP}
Suppose that $\Skr$ is of the form
\bq\label{eikp2}
\Sku(x) = w_+(x) \Eikp + \Skp(\ur)(x)
\eq
where $w_+(x) > 0$ for all $x$, and the PDE  $\Skp$ is elliptic.
Suppose that $$F(u)(x) \ge F(v)(x)$$ for all $x\in V\setminus\partial V$.
Consider the vertex set $\Ww$ and the quantity $M$ as in Definition~\ref{defn:M} and suppose that $M > 0$. Then we have that 
if $x \in \Ww\cap (V\setminus\partial V)$ and $z$ is an active neighbor of $x$, then $z\in \Ww$.

\end{lemma}
\begin{proof}
Choose a vertex $x\in \Ww\cap (V\setminus\partial V)$.  Because  $u-v$ has a local maximum at $x$ we have  $\grad u(x) \le \grad v(x)$. Therefore,  by the ellipticity of $\Skp$~\eqref{Fell} we have 
\[
\Skp(\ur)(x) \le \Skp(v)(x).
\]
Using our hypothesis on $F$ we conclude 
\[
\Eikp \ge \norm{\grad \vr(x)}^+,
\]
and because $z$ is an active neighbor,
\[
w_{xz} ( \ur(z) - \ur(x)) = \Eikp \ge  \norm{\grad \vr(x)}^+ \ge w_{xz} ( v(z) - v(x)),
\]
where the last inequality holds by the definition of the Eikonal operator.  This last result implies
\[
\ur(z) - v(z) \ge \ur(x) - v(x).
\]
But from the definition of $\Ww$ this last inequality must be an equality.   So $z$ is also in $\Ww$.
\end{proof}

\subsection{Proofs: uniqueness for eikonal equations}\label{sec:uniqueEik}

\begin{theorem}
Consider the Dirichlet problem~\eqref{D} for the positive eikonal equation ~\eqref{eikp1} on a graph $G$ with connected nonempty boundary $\partial \V$.  Suppose $\ur,\vr$ are solutions of the Dirichlet problem for~$\Skr$.  Then 
\[
u(x) \le v(x) \text{ for } x \in \partial \V
\implies 
u(x) \le v(x) \text{ for } x \in \V.
\]
\end{theorem}

\begin{remark} 
A similar result holds for negative eikonal equations.
\end{remark} 

\begin{proof} 
Define $M$ and $\Ww$ as in Definition~\ref{defn:M}. Suppose $M > 0$.
We also define
\[
C=\max_{x\in \Ww} \ur(x),
\qquad
\GG=\{x\in \Ww \mid \ur(x)=C\} . 
\]
Consider $x \in \GG$ and choose $z$ to be an active neighbor of $x$. 
From Lemma~\autoref{lemHP}, since~\eqref{eikp1} can be written in the form~\eqref{eikp2}, we conclude that  $z \in \Ww$.

Because $u$ is a solution, we have
\begin{align*}
0 &= w_+(x) \Eikp + f(x,\grad u) - g(x)
\\
&\le w_+(x) \Eikp + w_0 \Eikp - g(x)
\end{align*}
by  property~\eqref{fprop2} of $f$.  Because $g$ is positive we must have
\[
\Eikp > 0.
\]
This means $u(z) > u(x) = C$ which is a contradiction to $z \in \Ww$.
\end{proof}

\subsection{Proof of uniqueness for $p$-harmonious  equations}

\begin{lemma}[Of Harnack type]\label{lemUC}
Let $\Skr$ be $p$-harmonious~\eqref{pharm2} and $x\in V\setminus\partial V$ be an
interior vertex.
If $\Skr(u)(x) = 0$, then either 
\bq\label{Enp}
\Eikn < 0 <  \Eikp
\eq
or
\[
\grad u(x) = \zb.
\]
\end{lemma}
Note that~\eqref{Enp} is equivalent to 
\bq\label{uconst}
\min u(\mathbf{y}(x))  < u(x) < \max u(\mathbf{y}(x)), 
\eq
which is the reason for describing the result as of Harnack type.
\begin{proof}

Assume $\grad u(x) \not= \zb$.
Immediately following from  the definition, we conclude that one of $\Eikn$ or $ \Eikp$ is nonzero. 
Suppose that $\Eikn \not = 0$ (the argument will follow in a similar way for the other case).

Applying~\eqref{fprop2} to~\eqref{pharm2} we obtain
\begin{align}
\label{a}\tag{i}
(w_0(x)+w_-(x))\Eikn &+ w_+(x) \Eikp
\le 0
\\
\label{b}\tag{ii}
0
\le 
w_-(x)\Eikn &+ (w_0(x)+w_+(x)) \Eikp.
\end{align}

If $\Eikn > 0$, then by definition $\Eikp > 0$ which contradicts~\eqref{a}.  So that we must have $\Eikn < 0$.  But then using~\eqref{b}, we obtain $\Eikp > 0$, thereby obtaining the result~\eqref{Enp}.
\end{proof}

\begin{theorem}
Consider the Dirichlet problem~\eqref{D} for the $p$-harmonious function $\Skr$  on a graph $G$ 
which is connected  to its nonempty boundary $\partial \V$.  
Suppose $\ur,v$ are solutions of the Dirichlet problem for~$\Skr$.  Then 
\[
u(x) \le v(x) \text{ for } x \in \partial \V
\implies 
u(x) \le v(x) \text{ for } x \in \V.
\]
\end{theorem}
\begin{proof} 
 Define the vertex set $\Ww$ and the quantity $M$ as in Definition~\ref{defn:M}.  Suppose $M > 0$.
Also define
\[
C=\max_{x\in \Ww} \ur(x),
\qquad
\GG=\{x\in \Ww \mid \ur(x)=C\}. 
\]

Consider $x \in \GG$ and choose $z$ to be an active neighbor of $x$.
Since the $p$-harmonious equation is also of the form~\eqref{eikp2}, we can apply Lemma~\ref{lemHP}, to conclude that  $z \in \Ww$.

Next,  we claim that $\grad \ur(x) =0$.  
Since $z\in \Ww$, by the definition of $C$ we have that  $u(z) \le C$. On the other hand, $\ur(x) = C$.
Since we assumed that $\ur$ is a solution, we can apply Lemma~\autoref{lemUC}, to conclude $\grad \ur(x) = 0$.  This implies  that  the  set $N(x) \subset 
\GG\subset\Ww$.  
It follows that we can  find a path from $x$ to $\partial \V$ which stays in $\GG$.  But we assumed that $u \le v$ on $\partial \V$ and that  $u > v$ on $\GG$, and hence a contradiction.\end{proof}

\section{Existence results}

The existence theory is somewhat delicate, since there are known examples where existence can fail.    In the special case of operators of the form 
\[
- \lambda u + F(u) 
\]
with $F$ elliptic, existence and uniqueness results were established in Theorem 7 of \cite{ObermanSINUM}.  Therefore in this section we consider equations involving no such dependence on $u$.

%
%

\subsection{Existence for homogeneous equations}

We consider homogeneous equations, which include the  Laplacian, the Infinity Laplacian, the Eikonal equations, and the other examples from~\autoref{sec:genform}.

\begin{theorem}
Let $f(x,\mathbf{p})$ be elliptic~\eqref{elliptic}  and homogeneous~\eqref{fprop2}.  Suppose that $f$ is continuous in $\mathbf{p}$ for each $x$.  Then the Dirichlet problem~\eqref{D} has a solution.
\end{theorem}

\begin{proof}
We use the Brouwer fixed point theorem: a continuous function from a convex, compact subset $K$ of Euclidean space to itself has a fixed point.

We identify the set of functions on the graphs  $C(G)$  with $\R^N$ and 
consider the 
set 
\[
K = 
\left\{ 
u \in C(G) \mid 
u(x) = g(x),  x\in \partial \V,
m \le u(x) \le M,  x \in V\setminus \partial \V
\right \},
\]
where
\[
m = \min_{x\in \partial \V} g(x), \qquad
M =  \max_{x\in \partial \V} g(x).
\]
Then $K$ is a convex and compact set.

Define the constant 
\[
L =  {\max_{x\in V\setminus\partial V,\, y\in \mathbf{y}(x)}   {w_0(x)}{w_{xy}}}\, ,
\]
and supposed that $L>0$.  Define the mapping $T:C(G) \to C(G)$ by
\begin{eqnarray}
T(u)(x) = &u(x) + \frac 1 L f(x,\grad u(x)), &\text{ for } x\in V\setminus\partial V\label{NAmap}
\\
T(u)(x)= &g(x), &\text{ for }x\in \partial V. \notag
\end{eqnarray}
We claim $T$ is continuous and that it takes $K$ to $K$.   The result will follow from this claim and the Brouwer fixed point theorem.
\par
The continuity of $T$ follows from the continuity of $f(x,\mathbf{p})$ in the
variable $\mathbf{p}$ for each fixed $x\in V \setminus\partial V$. To see this,
enumerate the vertices
$$\{x_1, x_2, \ldots, x_n, x_{n+1}, \ldots x_{N}\},$$
where the first $n$ vertices are interior vertices and the last $N-n$ are boundary
vertices. The map $T\colon \mathbb{R}^N\mapsto\mathbb{R}^N$ is given by an
expression of the form
$$(Tu)_i= u_i+\frac{1}{L} f(x_i, \{u_j-u_i\}_j),$$
where the second argument of $f$ is just $\nabla u(x_i)$ for $1\le i \le n$ and
by 
$$ (Tu)_i= g(x_i),$$
which is a constant value in $u$, for $n<i\le N$.

\par

Note that from~\eqref{fprop2} we have for interior vertices $x$
\[
w_0(x) {\norm{\mathbf{p}}^{-}}\le f(x,\mathbf{p}) \le w_0(x)  {\norm{\mathbf{p}}^{+}},
\]
so by the definition of $L$ and the positive eikonal operator~\eqref{EikPDefn}
\begin{eqnarray*}
\frac 1 {L} f(x,\grad u(x)) &\le&  \frac 1 L {w_0(x)}  {\norm{\grad u(x)}^{+}}\\
&=&  \frac 1 L w_0(x)     \max_{1\le i\le d(x)} w_{xy_i}    ( \ur(y_i) - \ur(x)) \\
&=&  \frac 1 L w_0(x)    \, \,   w_{xy_{i_0}}   ( \ur(y_{i_0}) - \ur(x)),  \text{   for some  } i_0\in[1,d(x)]\\
&\le&  \frac 1 L w_0(x)    \, \,   w_{xy_{i_0}}   ( \ur(y_{i_0}) - \ur(x))^+\\
&\le &  \frac 1 L w_0(x)    \, \,   w_{xy_{i_0}}    \max_{y\in N(x)}    ( \ur(y) - \ur(x))\\
&\le &\max_{y\in N(x)} \left \{ u(y) - u(x) \right \},
\end{eqnarray*}
where we have used the set $N(x)$ in the last two lines in case $ {\norm{\grad u(x)}^{+}}<0$.
Similarly by  using ~\eqref{EikNDefn}, we get 
\[
\frac 1 {L} f(x,\grad u(x)) \ge 
 \min_{y\in N(x)}  
\left \{ u(y) - u(x) \right \}.
\]
We conclude that 
\[
\min_{y\in N(x)} u(y) \le T(u)(x) \le \max_{y\in N(x)} u(y), 
\]
and so 
\[
m \le T(u)(x) \le M, \quad \text{ for all } x \in V.
\]
Then the claim follows when $L>0$. When $L=0$ we must have
$\omega_0(x)=0$,  and thus $f(x, \mathbf{p})=0$,   for all interior vertices $x\in V\setminus\partial V$. In this case, any extension of $g$ from $\partial V$ to
$V$ is a solution to the Dirichlet problem ~\eqref{D}.
\end{proof}

\subsection{Existence for eikonal equations}
We generalize the fact that the distance function to the set $\partial \V$ is the solution of the eikonal equation, Example~\ref{ex:eikonal}, to construct solutions of more general eikonal equations.

\begin{theorem}
The Dirichlet problem
\[
f(x,u,\grad u(x)) =
\begin{cases}
 \norm{\grad u(x)}^+ -  h(x), &  x \in V\setminus \partial V,
\\
u(x) - g(x), & x \in \partial \V.
\end{cases}
\]
for $h(x) > 0$, has a solution.
\end{theorem}

\begin{proof}
First notice that when $g = 0$ and $h = 1$ for all $x$, the solution is given by  the (negative) minimal distance along the graph to the boundary vertices
 \[
 u(x) = - \min_{y\in \partial \V} \dist(x,y).
 \]
It follows that $\Eikp = 1$, since $w_{xy}(u(y) - u(x)) = 1$ when $y$, the active neighbour in the minimum, is the next point on the minimizing path which achieves the distance to the boundary. 

Next, we can assume that $g \ge 0$, by replacing $u$ by $u+c$ for an appropriate constant.  We can eliminate the nonzero Dirichlet values as follows.  Suppose $x\in \partial \V$, with $g(x) = g_1$.  Then add a single neighbor $ y_1\not= x$ so that
$\mathbf{y}(y_1)=\{x\}$ with $w_{xy_1} = 1/g_1$.  Consider the new equation on the extended graph, where $x$ is now
an interior vertex  and  $y_1$ is a boundary vertex,  and we set $g(y_1) = 0$.  In this way we have constructed a new Dirichlet problem for the eikonal equation with zero Dirichlet values,  which yields a solution of the original problem.

Finally, for nonconstant $h$, we can consider the graph with  $w_{xy}$ replaced by $w_{xy}h(x)$.
Although this graph may have non-symmetric weights, but this poses no additional difficulties.  
Now we can apply the previous existence result to the rescaled equation.
\end{proof}

\section{Connections with Finite difference approximations}\label{sec:finitediff}
This section focuses on how PDEs on graphs can be obtained as elliptic finite difference approximations to elliptic partial differential equations. 
The natural class of difference schemes for fully nonlinear elliptic equations is the class of monotone (elliptic) finite difference schemes, because they respect the comparison principle.  Such schemes can be shown to converge to the unique viscosity solution of the underlying PDE~\cite{BSnum}.   A systematic method for building monotone difference schemes was developed in~\cite{ObermanSINUM}, where the class of elliptic difference schemes (which agrees with elliptic PDEs, below) was identified.

We build elliptic finite difference approximations of the PDE operators by using Taylor approximations.  For example,  for a smooth function $u(x)$ we have
\begin{align*}
u(x+h) &= u(x) + h u'(x) + \frac{h^2}{2} u''(x) + \frac{h^3}{3!} u'''(x) + \bO(h^4)\\
\end{align*}
and
\begin{align*}
u(x-h) &= u(x) - h u'(x) + \frac{h^2}{2} u''(x) - \frac{h^3}{3!} u'''(x) + \bO(h^4).
\end{align*}
Averaging these equations gives the familiar finite difference expression
\[
u''(x) = \frac{u(x+h) -2 u(x)+u(x-h) }{h^2} + \bO(h^2).
\]
First order expressions for $u_x$ are also easily obtained.  Combining these gives
\begin{align*}
\abs{u_x}& = \frac{1}{h} \max( u(x+h) - u(x), u(x-h) - u(x)) +  \bO(h)\\
\end{align*}
and
\begin{align*}
-\abs{u_x} &= \frac{1}{h} \min ( u(x+h) - u(x), u(x-h) - u(x)) +  \bO(h),
\end{align*}
which are monotone schemes for the operators $\abs{u_x}, -\abs{u_x}$, respectively~\cite{ObermanSINUM}.
Likewise, the approximation for the infinity Laplace operator (\ref{infinitylap})
 is given by 
\bq\label{ILallBall}
-\Delta_\infty u( y) =
\frac{1}{r^2}
\left(
2u(x) - \min_{x\in \partial B_r(y)} u(x) - \max_{x\in \partial B_r(y)} u(x)
\right) + \bO(r),
\eq
which comes from the fact that 
\[
y^- \equiv \arg\min_{x\in \partial B_r(y)} u(x) =  x-r \frac{\grad u}{\abs{\grad u}}+ \bO(r^2)
\]
 and likewise for the maximum
 \[
y^+ \equiv \arg\max_{x\in \partial B_r(y)} u(x) =  x+ r \frac{\grad u}{\abs{\grad u}}+ \bO(r^2).
\]
Thus, we can write
\[
 u(y^+) = u(x) + r \abs{\grad u(x)} + \frac {r^2}{2\abs{\grad u}^2} \langle D^2 u(x)  \grad u(x) ,\grad u(x)\rangle+ \bO(r^3)
\]
and likewise
\[
 u(y^-) = u(x) - r \abs{\grad u(x)} + \frac {r^2}{2\abs{\grad u}^2} \langle D^2 u(x)  \grad u(x) ,\grad u(x)\rangle + \bO(r^3)
\]
so that we have
\[
\frac{u(y^+) + u(y^-)}{2}- u(x) =     \frac{1}{2}\, \langle D^2 u(x)  \frac{\grad u(x)}{|\grad u(x)|} ,\frac{\grad u(x)}{|\grad u(x)|} \rangle\, r^2+ \bO(r^3)
\]
which leads to the expression~\eqref{infinitylap}.

 The discretization of the expression onto a grid is less accurate, because of the lack of directional resolution, but using a wide stencil grid, a consistent, convergent scheme can be obtained~\cite{ObermanILnum}. In fact, the expression~\eqref{ILallBall} is second order accurate for smooth functions~\cite{ObermanpLap}.

Another example is the equation for the convex envelope~\cite{ObermanConvexEnvelope}. 
 The finite difference schemes for this equation lead to a PDE on a graph~\cite{ObermanEigenvalues,ObermanCENumerics}.
Neither the Obstacle problem for the convex envelope, nor the Dirichlet problem for the convex envelope~\cite{ObermanCED}  are proper. 
The PDE has the form $F(D^2u) = \lambda_1(D^2u)$, where $\lambda_1$ is the smallest eigenvalue of the Hessian $D^2u$.  This is discretized using
\[
\ME(u)(x)  = \min_{\norm{v} = 1} \frac{ u(x + hv) - 2u(x) + u(x-hv)}{h^2} + \bO(h^2).
\]
Consistency follows from the fact that the smallest eigenvalue of a symmetric matrix $M$ is given by $\min_{\norm{v} = 1} \langle M \cdot v, v\rangle$.  Applying this fact to the Hessian matrix, and using the Taylor series computation gives the result.

\section{Conclusions}
We have proved uniqueness and existence results for a wide class of uniformly elliptic and degenerate elliptic PDEs on graphs.
The approach combined viscosity solutions type techniques (the comparison principle) with connectivity properties of the graph to establish uniqueness results.  Existence results were established using fixed point theorems, or, in the case of generalized distance functions, explicit solutions.

In the future, we hope to solve the discrete equations on unstructured graphs.  In particular, we hope to implement fast solvers.

In addition the existence theory could be extended: natural analogues of the existence theory for viscosity solutions (barriers, Perron's method) could be studied. 

\bibliographystyle{alpha}
\bibliography{UniqueSchemes}

\end{document}